\documentclass[12pt,a4paper,psamsfonts]{amsart}
\usepackage{amssymb,amscd,amsxtra,calc}
\usepackage{cmmib57}

\theoremstyle{plain}
    \newtheorem{thm}{Theorem}[section]
    
    \newtheorem{claim}[thm]{Claim}
    \newtheorem{corollary}[thm]{Corollary}
    
    \newtheorem{lemma}[thm]{Lemma}

    \newtheorem{theorem}[thm]{Theorem}

\theoremstyle{definition}
    \newtheorem{definition}[thm]{Definition}

    \newtheorem{remark}[thm]{Remark}
\theoremstyle{remark}

    \newtheorem{setup}[thm]{}

\newcommand{\C}{\mathbb{C}}

\newcommand{\PP}{\mathbb{P}}
\newcommand{\Q}{\mathbb{Q}}

\newcommand{\R}{\mathbb{R}}

\newcommand{\Z}{\mathbb{Z}}

\newcommand{\OO}{\mathcal{O}}

\newcommand{\alb}{\operatorname{alb}}
\newcommand{\Aut}{\operatorname{Aut}}

\newcommand{\Exc}{\operatorname{Exc}}
\newcommand{\Gal}{\operatorname{Gal}}

\newcommand{\id}{\operatorname{id}}

\newcommand{\NE}{\overline{\operatorname{NE}}}
\newcommand{\Nef}{\operatorname{Nef}}
\newcommand{\NS}{\operatorname{NS}}

\newcommand{\Proj}{\operatorname{Proj}}
\newcommand{\rank}{\operatorname{rank}}
\newcommand{\Sing}{\operatorname{Sing}}
\newcommand{\SL}{\operatorname{SL}}
\newcommand{\Spec}{\operatorname{Spec}}
\newcommand{\Supp}{\operatorname{Supp}}

\newcommand{\Alb}{\operatorname{Alb}}

\begin{document}

\title[automorphism group]{
Algebraic varieties with automorphism groups of maximal rank}

\author{De-Qi Zhang}
\address
{
\textsc{Department of Mathematics} \endgraf
\textsc{National University of Singapore,
10 Lower Kent Ridge Road,
Singapore 119076}
}
\email{matzdq@nus.edu.sg}

\begin{abstract}
We confirm, to some extent, the belief that a projective
variety $X$ has the largest number (relative to the dimension of $X$) of
independent commuting automorphisms of positive
entropy only when $X$ is birational to a complex torus or a quotient of a torus.
We also include an addendum to an early paper \cite{CY3}
though it is not used in the present paper.
$$\text{\it Dedicated to the memory of Eckart Viehweg.}$$
\end{abstract}


\thanks{The author is supported by an ARF of NUS}

\maketitle

\section{Introduction}

We work over the field $\C$ of complex numbers.
We consider an automorphism group $G \le \Aut(X)$
of positive entropy on a compact complex K\"ahler manifold or normal projective variety
$X$. Our belief is: $X$ has the largest number
(relative to the dimension of $X$) of
independent commuting automorphisms of positive
entropy only when $X$ is a complex torus or a quotient of a torus.
We confirm this, to some extent, in Theorems \ref{ThA} and \ref{ThC}.
Our approach is conceptual and classification free.
See \cite{CZ} and \cite{CY3} for the case of threefolds or minimal varieties.

For an automorphism $g \in \Aut(X)$, its (topological) {\it entropy} $h(g) = \log \rho(g)$ is defined
as the logarithm of the {\it spectral radius} $\rho(g)$ of its action on the cohomology:
$$\rho(g) := \max \{|\lambda| \, ; \, \lambda \,\,\, \text{is an eigenvalue of} \,\,\,
{g^*}_{| \oplus_{i \ge 0} \, H^i(X, \C)}\}.$$
By the fundamental work of Gromov and Yomdin, the above definition is equivalent to the original definition
for automorphisms on compact K\"ahler manifolds or $\Q$-factorial projective varieties
(cf. \cite{Gr}, and also \cite[\S 2.2]{DS} and the references therein).

By the surface classification, a (smooth) compact complex surface $S$ has some $g \in \Aut(S)$ of positive entropy only if
$S$ is either the projective plane blown up in at least $10$ points,
or obtained by blowing up
some $g$-periodic orbits on a complex torus, $K3$ surface or Enriques surface (cf. \cite{Ca} for more details).
See \cite{JDG} for a similar phenomenon in higher dimension.

In their very inspiring paper \cite[Theorem 1]{DS}, Dinh and Sibony have proved the following
(cf. \cite{Tits} for its generalization to solvable groups).

\begin{theorem}\label{DSThm} {\rm (cf. \cite[Theorem 1.1]{DS})}
Let $X$ be a compact complex K\"ahler manifold of dimension $n \ge 2$ and $G \le \Aut(X)$ a commutative subgroup.
Then $G = N(G) \times G_1$ where $N(G)$ consists of all elements in $G$ of null entropy and is a subgroup of $G$,
and $G_1$ is a free abelian group of {\rm rank} $r = r(G_1) \le n-1$ and with $g_1$ of positive
entropy for all $g_1 \in G_1 \setminus \{\id\}$. Further, if $r = n-1$ then $N(G)$ is finite.
\end{theorem}

There do exist examples of $n$-dimensional complex tori, Calabi-Yau varieties and
rationally connected varieties $X$ admitting {\it maximal} rank
symmetries $\Z^{\oplus n-1} \cong G \le \Aut(X)$
with every element in $G \setminus \{\id\}$ being of positive entropy;
cf. \cite[Example 4.5]{DS} and \cite[Example 1.7]{CY3}.
All these examples are quotients of tori, as expected (cf. Theorem \ref{ThA} below).
Indeed, it is known that $\SL_n(\Z)$ includes a free abelian group $G$ of rank $n-1$
which has a natural faithful action on the complex $n$-torus $X := E^n$ with $g$ of positive
entropy for every $g \in G \setminus \{\id\}$, where $E$ is an elliptic curve
(cf. \cite[Example 4.5]{DS} for details and references therein).

Our standing assumptions are now (i) $G$ is commutative, (ii) every non-trivial element of $G$
has positive entropy, and (iii) rank $r(G) = \dim X - 1$.

Theorem \ref{ThA} below and Theorem \ref{ThC} in \S 2 are our main results. In Theorem \ref{ThC}(3)
a stricter restriction will be imposed on the $Y$
of Theorem \ref{ThA}(2).

For the definitions of {\it Kodaira dimension}
$\kappa(X)$ and singularities of {\it terminal}, {\it canonical} or {\it klt type},
we refer to \cite{KMM} or \cite[Definitions 2.34 and 7.73]{KM}.
A subvariety $Y \subset X$ is called $G$-{\it periodic}
if $Y$ is stabilized (set theoretically) by a finite-index subgroup of $G$.
Denote by $q(X) := h^1(X, \OO_{X})$ the {\it irregularity} of $X$.
A projective manifold $Y$ is a $Q$-{\it torus}
if $Y = A/F$ for a finite group $F$ acting freely on an abelian variety $A$.

\begin{theorem}\label{ThA}
Let $X$ be an $n$-dimensional normal projective variety
and let $G := {\Z}^{\oplus n-1}$ act on $X$ faithfully
such that every element of $G \setminus \{\id\}$ is of positive entropy.
Then the following hold:

\begin{itemize}
\item[(1)]
Suppose that
$\tau : A \to X$ is a $G$-equivariant
finite surjective morphism from an abelian variety $A$. Then $\tau$
is \'etale outside a finite set $($hence $X$ has only quotient
singularities and is $klt)$; $K_X \sim_{\Q} 0$ $(\Q$-linear equivalence$)$;
no positive-dimensional proper subvariety $Y \subset X$ is $G$-periodic.

\item[(2)]
Conversely, suppose that
no positive-dimensional $G$-periodic
subvariety $Y \subset X$ is either fixed $($point wise$)$ by a finite-index subgroup of $G$,
or is a $Q$-torus with $q(Y) > 0$,
or has $\kappa(Y) = -\infty$.
Suppose
also one of the following two conditions.

\item[(2a)]
$n = 3$, and $X$ is $klt$.
\item[(2b)]
$n \ge 3$, and $X$ has only quotient singularities.

Then $X \cong A/F$ for a finite group $F$ acting freely outside a finite set of an abelian variety $A$.
Further, for some finite-index subgroup $G_1$ of $G$, the action of $G_1$ on $X$ lifts to
an action of $G_1$ on $A$.
\end{itemize}
\end{theorem}

As a consequence of Theorems \ref{DSThm} and \ref{ThA}, we have:

\begin{corollary}\label{CorB}
Let $X$ be a normal projective variety of dimension $n \ge 3$ with
only
quotient singularities,
and let $G := {\Z}^{\oplus r}$ act on $X$ faithfully for some
$r \ge n-1$
such that every element of $G \setminus \{\id\}$ is of positive entropy.
Then $r = n-1$. Further,
$X \cong A/F$ for a finite group $F$ acting freely
outside a finite set of an abelian variety $A$,
if and only if
$X$ has no positive-dimensional $G$-periodic proper subvariety $Y \subset X$.
\end{corollary}

The proof of Theorem \ref{ThC} gives the following, which was also essentially proved
in \cite{CY3}.

\begin{corollary}\label{CorD}
Let $X$ be a normal projective variety of dimension $n \ge 3$
and let $G := {\Z}^{\oplus r}$ act on $X$ faithfully for some
$r \ge n-1$
such that every element of $G \setminus \{\id\}$ is of positive entropy.
Then $r = n-1$.
Suppose further that both $X$ and $(X, G)$ are minimal in the sense of $\ref{min}$,
and either $X$ has only quotient singularities, or $X$ is a $klt$ threefold.
Then $X \cong A/F$ for a finite group $F$ acting freely
outside a finite set of an abelian variety $A$.
\end{corollary}

\begin{remark}\label{rThA}
(1) In Theorems \ref{ThA} (2) and \ref{ThC},
we need to assume that $\dim X \ge 3$
which is {\it used at the last step} \ref{pfThA} to show the vanishing of
the second Chern class $c_2(X)$.
In fact, inspired by the comment of the referee, one notices that
a complete intersection $X \subset \PP^2 \times \PP^2$
of two very general hypersurfaces of type $(1, 1)$ and $(2, 2)$
is a $K3$ surface (called Wehler's surface) of Picard number two,
$X$ has an automorphism $g$ of entropy $2 \log(2 + \sqrt{3}) > 0$, and
$X$ contains no $(-2)$-curve, so there is no $g$-periodic curve on $X$;
cf.~\cite[Theorem 2.5, Proposition 2.6]{We}, \cite[Lemma 2.1, Proposition 2.5]{Si}.
Thus, both $X$ and the pair $(X, \langle g \rangle)$ are minimal in the sense of \ref{min}.
However, $X$ is not birational to the quotient of a complex torus, 
because a (smooth) projective $K3$ surface $X$ birational to the quotient of a complex torus
has the transcendental lattice of rank $\le$
(that of a complex $2$-torus), i.e., $\le 5$, and hence has Picard number $\ge (h^2(X, \C) - 5)$
which is $17$.

(2) In Theorem \ref{ThA} (2) (resp. Theorem \ref{ThC}), we can weaken the assumption on $G$
as a condition on $G^* := G_{|\NS_{\R}(X)}$ (cf. \cite[Theorem 4.7]{DS}):

\par \vskip 1pc \noindent
 ``$G^* \cong \Z^{\oplus n-1}$ (resp. $G^* \cong \Z^{\oplus n-1}$,
$G$ is virtually solvable)
and every element of $G^* \setminus \{\id\}$
is of positive entropy."

\par \vskip 1pc \noindent
But we need also to replace the last sentence

\par \vskip 1pc \noindent
``Further, \dots of $G_1$ on $A$." in Theorems \ref{ThA} (2) and \ref{ThC} as:

\par \vskip 1pc \noindent
``Further, the action of $G$ on $X$ lifts to an action of a group
$\widetilde{G}$ on $A$ with $\widetilde{G}/\Gal(A/X) \cong G$."

\par \vskip 1pc \noindent
The virtual solvability of $G$ is used in the middle of Claim \ref{c1} and end of the proof of Lemma \ref{ample4}.
\end{remark}

Our bimeromorphic point of view, in terms of the minimality assumption in
\ref{CorD} and \ref{ThC},
towards the dynamics study seems natural,
since one may blow up some Zariski-closed and $G$-stable proper subset of $X$ 
(if such subset exists) to get another pair
$(X', G)$ which is essentially the same as the original pair $(X, G)$.

The very starting point of our proof is the existence of enough
nef eigenvectors $L_i$ of $G$,
due to the fundamental work of Dinh-Sibony \cite{DS}. The minimal model program
(cf. \cite{KM}) is used with references provided for non-experts. Our main contribution lies
in Lemmas \ref{ample3} and \ref{ample4} where we show that the pair
$(X, G)$ can be replaced with
an equivariant one so that $H:= \sum_{i=1}^n \, L_i$ is an ample divisor.
To conclude, we prove a result of Hodge-Riemann type for singular varieties
to show the vanishing of Chern classes $c_i(X)$ ($i = 1, 2$),
utilizing $H^{n-i} \cdot c_i(X) = 0 = H^{n-2} \cdot c_1(X)^2$.
Then we use the characterization of \'etale quotient of a complex torus as
the compact K\"ahler manifold $X$ with vanishing Chern classes $c_i(X)$ ($i = 1, 2$)
(cf. \cite[\S 1]{Be}) and its generalization to singular varieties (cf. \cite{SW}).

For a possible generalization of the proof to a K\"ahler $n$-fold $X$,
we remark that the restriction `$\rank r(G) = n-1$' implies that $X$ is either Moishezon and hence projective,
or has algebraic dimension $a(X) = 0$ (cf. \cite[Theorem 1.2]{Tits}).
Thus the case $a(X) = 0$ remains to be treated. See a related remark in
\cite[\S 3.6]{Ca06}.

\par \vskip 1pc \noindent
{\bf Acknowledgement.}
I would like to thank
the organizers of the
Workshop on Higher Dimensional Algebraic Geometry,
29 March - 2 April 2010, National Taiwan University,
for the opportunity to talk about the results of this paper,
and the referee for many constructive suggestions to improve the paper.

\section{Proof of Theorem \ref{ThA}}

In this section we prove Theorem \ref{ThA} in the Introduction and its slightly
generalized version
Theorem \ref{ThC} below.

We use the {\bf terminology} and {\bf notation} in \cite{Hart}, \cite{KMM} and \cite{KM}.
By $G_{|Y}$, we mean that
there is a natural (from the context) induced action of $G$ on $Y$.

\begin{definition}\label{min}
A normal projective variety $X$ is {\it minimal} if it has only
$klt$ singularities (cf. \cite[Definition 2.34]{KM}) and the {\it canonical divisor} $K_X$ is nef.
Here a divisor $D$ is {\it nef} if the intersection $D \cdot C = \deg(D_{|C}) \ge 0$
for every curve $C$ in $X$. It is known that a quotient singularity is $klt$
and the converse is true in dimension two (cf. \cite[Propositions 5.20 and 4.18]{KM}).

Let $G \le \Aut(X)$ be a subgroup.
A pair $(X, G)$ is {\it minimal} if: for every
finite-index subgroup $G_1 \le G$, every
$G_1$-equivariant birational morphism $\sigma: X \to X_1$
onto a variety $X_1$ so that the pair $(X_1, \Delta)$ is $klt$ for some
effective (boundary) $\R$-divisor $\Delta$ (cf. \cite[Definition 2.34]{KM}), is an isomorphism.
Here $\Delta$ is not required to be $G_1$-stable.

We remark that every fibre of such $\sigma$
is rationally chain connected by \cite[Corollary 1.5]{HM}. Thus every irreducible component of
the {\it exceptional locus} $\Exc(\sigma) \subset X$
(the subset of $X$ along which $\sigma$ is not isomorphic) is $G$-periodic and
uniruled.

The two minimality definitions above are slightly different from the ones in \cite{CY3}.

Although the Minimal Model Program predicts the existence
of a minimal model $X_m$ for every non-uniruled projective variety $X$
and such existence is a theorem now for varieties of general type
(cf. \cite{BCHM}),
it is much harder to prove the existence of a minimal pair $(X_m, G)$ for any given pair $(X, G)$,
because the regular action of $G$ on $X$ induces a priori only a birational action of $G$ on $X_m$.
\end{definition}

The result below is more precise than \ref{ThA}(2)
in terms of the restriction on $Y$ in \ref{ThC}(3) and is applicable
under
the Good Minimal Model Program which predicts that \ref{ThC} (4) and (5)
are always true (cf. \cite{BCHM} for its recent breakthrough, Remark \ref{rThC}).

\begin{theorem}\label{ThC}
Let $X$ be a normal projective variety of dimension $n \ge 3$,
and let $G := {\Z}^{\oplus r}$ act on $X$ faithfully for some
$r \ge n-1$
such that every element in $G \setminus \{\id\}$ is of positive entropy.
Then $r = n-1$.
Suppose further the following five conditions.

\begin{itemize}
\item[(1)]
Either $X$ has only quotient singularities, or is a $klt$ threefold.
\item[(2)]
The pair $(X, G)$ is minimal in the sense of $\ref{min}$.
\item[(3)]
No positive-dimensional $G$-periodic
subvariety $Y \subset X$ is either fixed $($point wise$)$ by a finite-index subgroup
of $G$, or a $Q$-torus with $q(Y) > 0$, or a rational curve.
\item[(4)]
Every projective manifold $Y$ with $\dim Y \le n-1$, $\kappa(Y) = -\infty$ and $q(Y) > 0$,
is uniruled.
\item[(5)]
Every projective manifold $Y$ with $\dim Y \le n-2$, $\kappa(Y) = 0$ and $q(Y) > 0$,
has a good minimal model in the sense of Kawamata.
\end{itemize}

Then $X \cong A/F$ for a finite group $F$
acting freely outside a finite set of an abelian variety $A$.
Further, for some finite-index subgroup $G_1$ of $G$, the action of $G_1$ on $X$ lifts to
an action of $G_1$ on $A$.
\end{theorem}

\begin{remark}\label{rThC}
\begin{itemize}
\item[(1)]
The Condition (5) in Theorem \ref{ThC} is always true when $n = \dim X \le 5$ (cf.~\cite[3.13]{KM}).
\item[(2)]
The condition (4) in Theorem \ref{ThC} is always true when $n \le 5$, by
applying Iitaka's $C_{k, r}$ to the {\it Albanese map} $\alb_Y : Y \to \Alb(Y)$.
Here the {\it Albanese} $\Alb(Y)$ of $Y$ is a complex torus and every holomorphic map
from $Y$ to a complex torus factors through the albanese map $\alb_Y$
(cf. \cite[Ch IV]{Ue}).
In particular, every subgroup $H \le \Aut(Y)$
induces a natural action of $H$ on $\Alb(Y)$ so that $\alb_Y$ is $H$-equivariant.
\end{itemize}
\end{remark}

We begin with two lemmas.

\begin{lemma}\label{G/H}
Let $G$ be a group and $H \lhd G$ a finite normal subgroup such that
$$G/H = \langle \overline{g_1} \rangle \times \cdots \times \langle \overline{g_r} \rangle
\cong \Z^{\oplus r}$$
for some $r \ge 1$ and $g_i \in G$.
Then there is an integer $s > 0$ such that
$G_1 := \langle g_1^s, \dots, g_r^s \rangle$ satisfies
$$G_1 = \langle g_1^s \rangle \times \cdots \times \langle g_r^s \rangle
\cong \Z^{\oplus r}$$
and it is a finite-index subgroup of $G$; further,
the quotient map $\gamma : G \to G/H$ restricts to an
isomorphism $\gamma_{|G_1} : G_1 \to \gamma(G_1)$ onto a finite-index subgroup of $G/H$.
\end{lemma}

\begin{proof}
We only need to find some $s > 0$ such that $g_i^s$ and $g_j^s$ are commutative to each other for all $i, j$.
Since $G/H$ is abelian, the commutator subgroup $[G, G] \le H$.
Thus the commutators $[g_1^t, g_2]$ ($t > 0$) all belong to $H$.
The finiteness of $H$ implies that $[g_1^{t_1}, g_2] = [g_1^{t_2}, g_2]$ for some $t_2 > t_1$,
which implies that $g_1^{s_{12}}$ commutes with $g_2$, where $s_{12} := t_2 - t_1$.
Similarly, we can find integers $s_{1j} > 0$ such that $g_1^{s_{1j}}$ commutes with $g_j$.
Set $s_1 := s_{12} \times \cdots \times s_{1r}$. Then $g_1^{s_1}$ commutes with
every $g_j$.
Similarly, for each $i$,
we can find an integer $s_{i} > 0$ such that $g_i^{s_i}$ commutes with $g_j$ for all $j$.
Now $s := s_1 \times \cdots \times s_{r}$ will do the job. This proves the lemma.
\end{proof}

\begin{lemma}\label{D=0}
Let $X$ be an $n$-dimensional projective variety, $H_i$ nef and big $\R$-Cartier divisors
and $D$ an $\R$-Cartier divisor such that $H_1 \cdots H_{n-1} \cdot D = 0$.
Then we have:
\begin{itemize}
\item[(1)]
$H_1 \cdots H_{n-2} \cdot D^2 \le 0$.
\item[(2)]
Suppose that $H_1, \dots, H_{n-2}$ are ample divisors.
Then $H_1 \cdots H_{n-2} \cdot D^2 = 0$ holds if and only if
$D \equiv 0$ $($numerically$)$, i.e., $D \cdot C = 0$ for every curve $C$ on $X$.
\end{itemize}
\end{lemma}

\begin{proof}
$H_1 \cdots H_{n-1} \cdot D = 0$ implies the assertion (1) by pulling back to a resolution of $X$
(cf. \cite[Corollary 3.4]{DS}).

We still need to prove the `only if' part of the assertion (2)
which will be done by induction on the dimension $n$.
When $n \le 2$,
the assertion (2) follows from the Hodge index theory (for surfaces).
Suppose that $n \ge 3$ and
$$H_1 \cdots H_{n-1} \cdot D = H_1 \cdots H_{n-2} \cdot D^2 = 0 .$$
Let $\sigma : X' \to X$ be Hironaka's resolution such that $-E$ is relatively ample
for some $\sigma$-exceptional effective divisor $E$.
Replacing $E$ by its small multiple, we may assume that $\sigma^*H_1 - E$ is ample
(cf. \cite[Proposition 1.45]{KM}).

For a curve $C$ on $X$, we take an ample irreducible divisor $A_1$ on $X'$
such that $\sigma(A_1)$ contains $C$.
Take very small $\varepsilon > 0$
such that $\varepsilon A_1 \le \sigma^*H_1 - E$.
Thus we can write $\sigma^*H_1 - E \equiv \sum_{i=1}^s r_i A_i$
where $r_i \in \R_{> 0}$ and $A_i$ are ample irreducible divisors.
Since $E$ is contracted by $\sigma$ and by the projection formula,
we have $\sigma_*E = 0$ and
$$E \cdot \sigma^*D_1 \cdots \sigma^*D_{n-1}
= \sigma_* E \cdot D_1 \cdots D_{n-1} = 0$$
for all Cartier divisors $D_i$.
Set $H_i' := \sigma^* H_i$, $D' := \sigma^* D$.
Then
\begin{equation}\label{D=0eq1}
\sum_{i=1}^s r_i A_i \cdot (\prod_{j=2}^{n-1} H_j') \cdot D' = (H_1' - E) \cdot (\prod_{j=2}^{n-1} H_j') \cdot D'
= (\prod_{j=1}^{n-1} H_j') \cdot D' = (\prod_{j=1}^{n-1} H_j) \cdot D = 0,
\end{equation}
\begin{equation}\label{D=0eq2}
\sum_{i=1}^s r_i A_i \cdot (\prod_{j=2}^{n-2} H_j') \cdot (D')^2 = (H_1' - E) \cdot (\prod_{j=2}^{n-2} H_j') \cdot (D')^2
= (\prod_{j=1}^{n-2} H_j') \cdot (D')^2 = (\prod_{j=1}^{n-2} H_j) \cdot D^2 = 0 .
\end{equation}
By the equality (\ref{D=0eq1}) above and since $A_i$ and $H_j$ are nef, we have (for all $i$):
\begin{equation}\label{D=0eq3}
({H_2'}_{|A_i}) \cdots ({H_{n-1}'}_{|A_i}) \cdot ({D'}_{|A_i}) = A_i \cdot H_2' \cdots H_{n-1}' \cdot D' = 0 .
\end{equation}
Hence $A_i \cdot H_2' \cdots H_{n-2}' \cdot (D')^2 \le 0$ by the assertion (1).
This together with the equality (\ref{D=0eq2}) above imply that for all $i$, we have
\begin{equation}\label{D=0eq4}
({H_2'}_{|A_i}) \cdots ({H_{n-2}'}_{|A_i}) \cdot ({D'}_{|A_i})^2 = A_i \cdot H_2' \cdots H_{n-2}' \cdot (D')^2 = 0 .
\end{equation}
Write $B_i := \sigma(A_i)$ which is birational to $A_i$.
By the equality (\ref{D=0eq3}) above,
$$\prod_{j=2}^{n-1} ({H_j}_{|B_i}) \cdot (D_{|B_i}) =
\prod_{j=2}^{n-1} ((\sigma^*{H_j})_{|A_i}) \cdot ((\sigma^*D)_{|A_i}) = 0 .$$
Similarly, the equality (\ref{D=0eq4}) above implies $\prod_{j=2}^{n-2} ({H_j}_{|B_i}) \cdot (D_{|B_i})^2 = 0$.
By the induction, $D_{|B_i} \equiv 0$. Note that $B_1 = \sigma(A_1)$ contains $C$.
Thus $D \cdot C = (D_{|B_1}) \cdot C = 0$.
The lemma is proved.
\end{proof}

We now prove Theorems \ref{ThA} and \ref{ThC}.

Let $\tau: \widetilde{X} \to X$ be a $G$-equivariant resolution due to Hironaka (cf. \cite[(2.0)]{Fu}
and the reference therein).
Applying the proof of \cite[Theorems 4.7 and 4.3]{DS} to the action of
$G$ on the pullback $\tau^*\Nef(X)$ of the nef cone $\Nef(X)$ (instead of the
K\"ahler cone of $\widetilde{X}$ there), we get
nef $\R$-Cartier divisors $\tau^*L_i$ ($1 \le i \le n$) on $\widetilde{X}$ (resp. $L_i$ on $X$)
as common eigenvectors of
$G$ such that the intersection $L_1 \cdots L_n \ne 0$ and the homomorphism
below is an isomorphism onto a spanning lattice (where we write
$g^*L_i = \chi_i(g)L_i$):
$$\begin{aligned}
\varphi : G  \, &\rightarrow \, (\R^{n-1}, +) \\
g \, &\mapsto \,  (\log \chi_1(g), \dots, \log \chi_{n-1}(g)).
\end{aligned}$$
Since $L_1 \cdots L_n = g^*(L_1 \cdots L_n) = \chi_1(g) \cdots \chi_n(g) \, L_1 \cdots L_n$,
we have
$$\chi_1 \cdots \chi_n = 1 .$$
Set
$$H := \sum_{i=1}^n L_i .$$
Then $H^n \ge L_1 \cdots L_n > 0$ and hence $H$ is a nef and big
$\R$-Cartier divisor.

\begin{lemma}\label{ample}
Under the assumption of Theorem $\ref{ThA}(2)$ or $\ref{ThC}$, the following are equivalent.
\begin{itemize}
\item[(1)]
$H$ is not ample.
\item[(2)]
$H^k \cdot Y = ({H}_{|Y})^k = 0$ for some proper subvariety $Y \subset X$ of dimension $k > 0$.
\item[(3)]
$X$ has a $G$-periodic proper subvariety $Y \subset X$ of dimension $k > 0$.
\end{itemize}
\end{lemma}

\begin{proof}
By Campana-Peternell's $\R$-divisor version of
Nakai-Moishezon ampleness criterion, the assertions (1) and (2) are equivalent.

For (3) $\Rightarrow$ (2), assume the assertion (3).
After $G$ is replaced by its finite-index subgroup $G_1$ and noting that
$r(G_1) = r(G) = n-1$,
we may assume that $Y$ is stabilized by $G$.
Note that for all $i_j$, we have
\begin{equation}\label{(*)}
L_{i_1} \cdots L_{i_k} \cdot Y = 0 .
\end{equation}
Indeed,
since $\varphi(G) \subset \R^{n-1}$ is a spanning lattice, $k \le n-1$, and $\chi_1 \cdots \chi_n
= 1$, we can choose $g \in G$ such that
$\chi_{i_j}(g) > 1$ for all $i_j$.
Acting on the left hand side of the equality (\ref{(*)}) (a scalar) with $g^*$ and noting that $g^*Y = Y$,
we conclude the equality (\ref{(*)}). This, in turn, implies that $H^k \cdot Y = 0$.

For (2) $\Rightarrow$ (3), assume that $H^k \cdot Y = 0$ as in the assertion (2).
Then $H_{|Y}$ is nef but not big.
Write $H = L + E$ with $L$ ample and $E$ effective
(cf. the proof of \cite[Proposition 2.61]{KM}).
Since $H_{|Y}$ is not big, $Y \subseteq \Supp E$.
Since $L_{i_j}$ are all nef,
$H^k \cdot Y = 0$ means $L_{i_1} \cdots L_{i_k} \cdot Y = 0$ for all
$i_j$. Since $L_{i_j}$ are
all $g^*$-eigenvectors, reversing the process,
we get $H^k \cdot g(Y) = 0$ and hence $g(Y) \subset \Supp E$
by the above reasoning.
The Zariski-closure $\overline{\cup_{g \in G} \, g(Y)}$
is $G$-stabilized and contained in $\Supp E$. Every irreducible
component of this closure is a positive-dimensional $G$-periodic proper subvariety of $X$.
This proves the assertion (3).
\end{proof}

\begin{setup}
In the proofs below, we will apply the Minimal Model Program to a pair $(X, D)$
of a variety $X$ and an effective $\R$-divisor $D$ where the pair has at worst $klt$ singularities
(cf. \cite[Definition 2.34]{KM}).
If $K_X + D$ is not nef, then there is a $(K_X + D)$-negative extremal ray $R = \R_{> 0}[\ell]$
of the {\it closed cone $\NE(X)$ of effective $1$-cycles}. Now the
cone theorem \cite[Theorem 3.7]{KM} gives rise to an
extremal contraction $\varphi : X \to Y$ to a normal variety $Y$ such that a curve $C$
is contracted by $\varphi$ to a point if and only if the class $[C]$ belongs to the extremal ray $R$.
There are exactly three types of such $\varphi$ (cf.~\cite[\S 3.7]{KM} for details):

\begin{itemize}
\item[(i)]
$\varphi$ is {\it divisorial}. It is a birational morphism whose exceptional locus
$\Exc(\varphi) \subset X$ (the locus where $\varphi$ is not
isomorphic) is a prime divisor.
\item[(ii)]
$\varphi$ is a {\it flip}. Then there is a {\it flipping} $X \dashrightarrow X^+$ which is
a rational map and isomorpohic in codimension one. Further, there is a birational
morphism $X^+ \to Y$ such that the composite $X \dashrightarrow X^+ \to Y$ coincides with $\varphi$.
In particular, there is a natural isomorphism between the Neron-Severi groups (with $\R$-coefficient)
of $X$ and $X^+$. Such $X^+$ is unique. Indeed,
$X^+ = \Proj \oplus_{m \ge 0} \varphi_*\OO_X(\lfloor m(K_X + D)\rfloor)$.
\item[(iii)]
$\dim Y < \dim X$. Then
$\varphi$ is called a {\it Fano fibration} so that the restriction ${-(K_X + D)}_{|F}$ of the anti-adjoint divisor
to a general fibre
$F$ of $\varphi$ is ample.
\end{itemize}
\end{setup}

\begin{lemma}\label{nefbig}
Under the assumption of Theorem $\ref{ThA}(2)$
or Theorem $\ref{ThC}$,
$K_X + sH$ is nef and big for some $($and hence all$)$ $s >> 1$.
\end{lemma}

\begin{proof}
Write $H = E/k + A_k$ with $A_k$ a general ample $\Q$-divisor and $E$ an effective $\R$-divisor
(cf. the proof of \cite[Proposition 2.61]{KM}). By the assumption, $X$ is $klt$.
Hence we can choose $k >> 1$ so that $(X, E/k + A_k)$ is
$klt$ (cf.~\cite[Corollary 2.35(2)]{KM}), where $A_k$ is replaced by $(\sum_{i=1}^m D_i)/m$ with
$D_i$ general members of $|m A_k|$ for some $m >> 1$.
Replace $H$ by $E/k + A_k$ for some large $k$ and fix an ample $\Q$-divisor $M$,
such that $K_X + H + M$ is nef and $klt$.

We may assume that $K_X + sH$ is not nef for any $s > 0$.
We now consider $K_X + H$, but $H$ may be replaced by $sH$ for some $s >> 1$.
By the cone theorem
(cf.~\cite[Theorem 3.7]{KM} or \cite[Corollary 3.8.2]{BCHM}),
there are only finitely many $(K_X + H)$-negative extremal rays $\R_{> 0} [\ell]$ in $\NE(X)$.
Replacing $H$ by a larger multiple, we may assume that all these $\ell$ satisfy $H \cdot \ell = 0$
(i.e., $L_i \cdot \ell = 0$ for all $i$) and $K_X \cdot \ell < 0$.
Since $L_i \cdot g^{-1}(\ell) = \chi_i(g) L_i \cdot \ell = 0$ and hence $H \cdot g^{-1}(\ell) = 0$, and
$K_X \cdot g^{-1}(\ell) = g^*K_X \cdot \ell = K_X \cdot \ell < 0$,
$g^{-1}(\ell)$ (also an extremal curve) satisfies the same conditions as $\ell$. So
these finitely many extremal rays $\R_{> 0} [\ell]$ are permuted,
and hence stabilized by a finite-index subgroup $G_1$ of $G$.
This $G_1$ will be used later on.

By \cite[Theorem 1.1(6)]{Fujino}
(which extends the result of Birkar),
there are some $1 \ge \lambda_0 > 0$
and extremal ray $\R_{> 0} [\ell_0]$ such that
$K_X + H + \lambda_0 M$ is nef,
$(K_X + H) \cdot \ell_0 < 0$ (and hence $H \cdot \ell_0 = 0$ and $K_X \cdot \ell_0 < 0$) and $(K_X + H + \lambda_0 M) \cdot \ell_0 = 0$.

Let $\varphi_{0} : X \to Y$ be the extremal contraction corresponding to the extremal ray $\R_{> 0} [\ell_0]$,
which is $G_1$-equivariant, where $G_1$ is as mentioned earlier on.
Thus every positive-dimensional irreducible component $F_i$
of $\Exc(\varphi_0)$ is $G_1$-periodic and hence $G$-periodic.

Suppose that $\varphi_{0}$ is a Fano fibration.
If $Y$ is not a point, this contradicts \cite[Lemma 2.10]{Tits}
since $r(G_1) = r(G) = \dim X - 1$ now.
If $Y$ is a point, then $X$ is Fano of Picard number one and the class of $-K_X$ is ample
and preserved by $G$. Thus a finite-index subgroup of $G$ is contained in
$\Aut_0(X)$ by the result of Lieberman and Fujiki
(cf. \cite[Proposition 2.2]{Li}, \cite[Theorem 4.8]{Fu}).
Hence $G$ is of null entropy, a contradiction.

Suppose that $\varphi_{0}$ is birational (i.e., divisorial or a flip).
Let $H_Y$ and $M_Y$ be the direct image on $Y$ of $H$
and $M$, respectively. Since $K_X + H + \lambda_0 M$ is perpendicular to $\ell_0$,
it is the pullback of $K_Y + H_Y + \lambda_0 M_Y$
(cf. \cite[Theorem 3.7(4)]{KM} or \cite[Corollary 3.9.1]{BCHM}), so the latter adjoint divisor on $Y$
(or the pair $(Y, H_Y + \lambda_0 M_Y)$)
is $klt$ because so is its birational pullback on $X$. Hence
for every irreducible component $F_i$ of $\Exc(\varphi_0)$,
the fibres of
${\varphi_0}_{| F_i} : F_i \to \varphi_0(F_i)$
are all rationally chain connected by \cite[Corollary 1.5]{HM}.
Thus $F_i$ are all $G$-periodic and uniruled
(and hence of Kodaira dimension $-\infty$).
This contradicts the assumption of Theorem \ref{ThA} (2), and Theorem \ref{ThC} (2) as well.

Thus $K_X + sH$ is nef for some $s > 0$ and the lemma follows since
$H$ is big (and nef).
\end{proof}

\begin{lemma}\label{ample2}
Under the assumption of Theorem $\ref{ThA}(2)$
or Theorem $\ref{ThC}$, $K_X + sH$ is ample for some $($and hence all$)$ $s >> 1$.
Moreover, $H_{|Y} \equiv 0$ $($numerically$)$ for every $G$-periodic subvariety $Y \subset X$,
and hence $(K_X)_{|Y}$ is ample.
\end{lemma}

\begin{proof}
By Lemma \ref{nefbig}, we may assume that $K_X + H$ is nef and big and $klt$
after replacing $H$ by its large multiple.
By the effective base-point freeness theorem
(cf. \cite[Theorem 3.9.1]{BCHM}),
there is a birational morphism $\psi : X \to Z$ onto a normal projective variety $Z$,
such that $K_X + H = \psi^*P$ for some ample divisor $P \subset Z$.
Write $H = E/k + A_k$ as in Lemma \ref{nefbig}. Thus every extremal ray
$\R_{\ge 0}[\ell] \subset \NE(X)$ contracted by $\psi$ is $(K_X + E/k)$-negative.
By the cone theorem, there are only finitely many such extremal rays.
Replacing $H$ by its large multiple, we may assume that such $\ell$ satisfies
$\ell \cdot H = 0 = \ell \cdot K_X$, the condition of which is preserved by $G$.
Thus we may assume that all such extremal rays are stabilized (resp. permuted)
by a finite-index subgroup $G_1$ of $G$ (resp. by $G$). In particular,
$\psi$ is $G_1$-equivariant;
$\Exc(\psi)$ and every positive-dimensional irreducible component $F_i$ of it
is $G$-periodic. Let $H_Z \subset Z$ be the direct image of $H$.
Then $K_X + H$ is the pullback of $K_Z + H_Z$ ($\sim_{\Q} P$)
and hence the latter adjoint divisor on $Z$ (or the pair $(Z, H_Z)$)
is $klt$ because so is its pullback on $X$.
As in Lemma \ref{nefbig}, each $F_i$ is $G$-periodic and uniruled, contradicting
the assumption of Theorem \ref{ThA} (2), and Theorem \ref{ThC} (2) as well,
unless $\psi : X \to Z$ is an isomorphism, i.e., $K_X + H$ is ample.

For the second assertion, we claim the following vanishing of the intersection:
\begin{equation}\label{(*2)}
((K_X + sH)_{|Y})^{k-1} \cdot H_{|Y} = (K_X + sH)^{k-1} \cdot H \cdot Y = 0
\end{equation}
where $k = \dim Y$.
Indeed, since $H = \sum_{i=1}^n L_i$, the above intersection number is
the summation of the following terms
$$
K_X^{k-1-t} \cdot L_{j_1}  \cdots L_{j_t} \cdot L_i \cdot Y
$$
where $0 \le t \le k-1 \le n-2$.
Now the vanishing of each term above can be verified as
in Lemma \ref{ample}, since $g^*K_X \sim K_X$ for $g \in G$.
The equality (\ref{(*2)}) above is proved.

The equality (\ref{(*2)}) and ampleness of $K_X + sH$ imply that
the scalar
$$((K_X + sH)_{|Y})^{k-2} \cdot (H_{|Y})^2$$
is non-positive
by Lemma \ref{D=0},
and hence is zero since $K_X + sH$ and $H$ are nef.
Thus $H_{|Y} \equiv 0$ by Lemma \ref{D=0}.
\end{proof}

\begin{lemma}\label{ample3}
Under the assumption of Theorem $\ref{ThA}(2)$, $H$ is ample.
\end{lemma}

\begin{proof}
Suppose the contrary that $H$ is not ample. Then by Lemma \ref{ample},
a subvariety $Y \subset X$ of positive-dimension $k$ is
$G$-periodic. We choose such $Y$ with
$k \in \{1, \dots, n-1\}$ minimal.
This $Y$ is stabilized by a finite-index subgroup $G_1$ of $G$.
Now the class of the ample divisor $(K_X)_{|Y}$ (cf.~Lemma \ref{ample2})
is fixed by the pullback of every $g \in G_1$.
Let $\tau: Y' \to Y$ be a $G_1$-equivariant desingularization.
Then $\tau^*((K_X)_{|Y})$ is nef and big on $Y'$ and its class is fixed by the pullback of every
$g \in G_1$. Thus ${G_1}_{|Y'} \le \Aut_0(Y')$ after
$G_1$ is replaced by a smaller finite-index subgroup of $G$ by the result of Lieberman and Fujiki
(cf.~\cite[Lemma 2.23]{JDG}, \cite[Proposition 2.2]{Li}, \cite[Theorem 4.8]{Fu}).

Suppose that the Kodaira dimension $(\kappa(Y) :=$) $\kappa(Y') \ge 1$.
Then $G$ (replaced by its finite-index subgroup)
acts trivially on the base of the
Iitaka fibration $Y' \dashrightarrow B$ (with $\kappa(Y') = \dim B$), by a classical
result of Deligne-Nakamura-Ueno \cite[Theorem 14.10]{Ue}.
Hence $G$ stabilizes a general fibre $Y'_b$ over a point $b \in B$.
By the minimality of $\dim Y = k > 0$, we have $\dim Y'_b = 0$ and hence $Y'$ is of general type
so that $\Aut(Y')$ (and hence $G_{|Y'}$)
are known to be finite; thus $Y$ is fixed (point wise) by
a subgroup of $G_{|Y} \le \Aut(Y)$, contradicting the assumption
of Theorem \ref{ThA}(2).

Therefore, we may assume that $\kappa(Y) \le 0$.
Thus $\kappa(Y) = 0$ by the assumption in Theorem \ref{ThA}(2),
which is stronger than the condition (3) in Theorem \ref{ThC}.

For $\Aut_0(Y')$, we shall use the terminology and results
in \cite[Theorem 3.12]{Li} or \cite[\S 2, Theorem 5.5]{Fu}.
If $\Aut_0(Y')$ has a positive-dimensional linear part, then
$Y$ is ruled (which contradicts $\kappa(Y) = 0$) by a classical result of Matsumura
or its generalization \cite[Proposition 5.10]{Fu};
thus we may assume that the linear part of $\Aut_0(Y')$ is trivial;
then the classical
Jacobi homomorphism $\Aut_0(Y') \to \Aut_0(\Alb(Y')) \cong \Alb(Y')$
has a finite kernel (cf. \cite[Theorem 3.12]{Li}, \cite[Theorem 5.5]{Fu}).
If ($\dim \Alb(Y') =$) $q(Y') = 0$, then ${G_1}_{|Y'} \le \Aut_0(Y') = (1)$
and $Y$ is fixed (point wise)
by $G_1$, contradicting the assumption of Theorem \ref{ThA}(2).

Thus we may assume that $q(Y') > 0$.
The singular locus $\Sing Y$ is clearly stabilized by $G_1$. By the minimality of $\dim Y = k > 0$,
$\Sing Y$ is empty or finite.
Since $\kappa(Y) = 0$, $K_{Y'} \sim_{\Q} D$ for some
effective $\Q$-divisor $D$. Since $g^*D \sim_{\Q} g^*K_{Y'} \sim K_{Y'} \sim_{\Q} D$
for $g \in G_1$, we have $g^*D = D$ for $\kappa(Y') = 0$.
Thus all irreducible components of
$D \subset Y'$ and their images on $Y$ are $G_1$-periodic.
The minimality of $\dim Y = k > 0$ implies that $D$ is contracted to a few points on $Y$
and also on $Y^n$, if we factor $Y' \to Y$ as $Y' \to Y^n \to Y$
with $Y^n \to Y$ the normalization.
Thus $K_{Y^n} \sim_{\Q} 0$, since it is the direct image of $K_{Y'}$ ($\sim_{\Q} D$).
Further, $Y^n$ has at worst canonical singularities since $K_{Y'}$
is the pullback of $K_{Y^n}$ plus an effective divisor $D$.

By the proof of \cite[Theorem B]{NZ} (cf. also \cite[\S 3, especially Proposition 3]{Be}),
there is a finite \'etale
morphism $F \times A \to Y^n$ such that $A$ is an abelian variety,
$F$ is a weak Calabi-Yau variety (and hence $q(F) = 0$) and
$$M := \Aut_0(Y^n)$$
lifts to a {\it split} action of
$\widetilde{M}$, with $\widetilde{M}/(\Gal((F \times A)/Y^n)) \cong M$,
on the product $F \times A$. Since $q(F) = 0$
and hence $\Aut_0(F) = (1)$ (cf. \cite[Lemma 4.4]{NZ}),
the connected algebraic group $\widetilde{M}$ acts on the factor $F$ trivially.
Thus $\widetilde{M}$ stabilizes all fibres $\{f\} \times A$, and hence ${G_1}_{|Y^n}$
and ${G_1}_{|Y}$ stabilize their images (i.e., quotients of tori $\{f\} \times A$),
which form a so called torus-quotient covering family of $Y^n$.

By the minimality of $\dim Y = k > 0$,
we have $\dim Y = \dim A$ ($\ge q(Y') > 0$), i.e., $\dim F = 0$.
Thus $Y^n$ is a $Q$-torus (by taking the Galois closure of $A \to Y^n$).
In particular, $Y^n$ is smooth and we may take $Y' = Y^n$.

By the assumption in Theorem \ref{ThA}(2), $Y$ is not a $Q$-torus with $q(Y) > 0$.
Hence $Y^n \ne Y$. Namely, $\Sing Y$ is a non-empty finite set.
Its inverse on $Y^n$ is stabilized
by ${G_1}_{|Y'}$. Further, the image on $\Alb(Y^n)$ of this (finite) inverse is stabilized
by the image of ${G_1}_{| Y^n}$ in $\Aut_0(\Alb(Y^n))$ ($=$ translations)
under the above Jacobi homomorphism. Hence such homomorphic image is trivial.
So ${G_1}_{| Y^n}$ is a finite group because the Jacobi homomorphism
has a finite kernel as mentioned earlier on. Thus $Y^n$ and hence $Y$ are fixed (point wise)
by a finite-index subgroup of $G$.
This contradicts the assumption in Theorem \ref{ThA}(2).
\end{proof}

\begin{lemma}\label{ample4}
Under the assumption of Theorem $\ref{ThC}$, $H$ is ample.
\end{lemma}

\begin{proof}
We use the argument and notation of Lemma \ref{ample3}.
The condition (3) in Theorem \ref{ThC} and the argument
in Lemma \ref{ample3} imply that $\kappa(Y) := \kappa(Y') = -\infty$.
We may assume that ${G_1}_{|Y'}$ is an infinite group,
otherwise, $Y'$ and hence $Y$ are fixed (point wise) by a finite-index subgroup of $G$,
contradicting the condition (3) in Theorem \ref{ThC}.

\begin{claim}\label{c1}
The irregularity $q(Y') > 0$.
\end{claim}

We prove the claim.
Suppose the contrary that $q(Y') = 0$. Then
$\Aut_0(Y')$ is an affine algebraic group. Let $\overline{G}_1 \le \Aut_0(Y')$
be the identity connected component of the closure of ${G_1}_{|Y'} \le \Aut_0(Y')$ in the Zariski topology
which is abelian because so is $G_1$.
By the minimality of $\dim Y = k > 0$,
the induced action of the affine algebraic group $\overline{G}_1$ on the normalization $Y^n$
of $Y$ has a Zariski-dense open orbit $\overline{G}_1 y$, a few isolated orbits
and no others.
Since $\overline{G}_1$ is abelian, it is solvable and has a fixed point
(by Borel's fixed point theorem), so it does have an isolated orbit in $Y^n$.
By \cite[Theorem and its Remark, p. 182]{HO},
$Y^n$ is a projective cone over a rational
homogeneous projective manifold which again has a fixed point by the induced action
of $\overline{G}_1$.
Thus at least one of the lines generating the
cone $Y^n$ is stabilized by $G_1$; the image of this line in $Y$ is a $G$-periodic rational curve,
contradicting the condition (3) in Theorem \ref{ThC}.
This proves the claim.

\par \vskip 1pc
By Claim \ref{c1}, $q(Y') > 0$ and hence $Y'$ is not rationally connected
(cf. \cite[Remark 5.1]{NZ}).
Then  $Y'$ is uniruled, by the condition (4) of Theorem \ref{ThC} and since $\kappa(Y') = -\infty$.

Note that the class of the ample divisor $(K_X)_{|Y}$ is preserved by $g^*$ ($g \in G_1 \le G$).
We now apply results in \cite[Theorem 4.18, Corollary 4.20]{IntS} (cf. also
\cite[Lemma 4.1]{nz2}). There is a `special maximal rationally connected fibration'
$Y \dashrightarrow Z$
onto a normal projective variety $Z$ such that
its graph $W = \Gamma_{Y/Z}$ is equi-dimensional over $Z$
(i.e., every fibre $W_z$ over a point $z \in Z$
is of pure-dimension equal to $\dim W - \dim Z$)
and
${G_1}_{|Y}$ descends to
${G_1}_{| Z}$ with $g^*H_Z \sim H_Z$ ($g \in G_1$) for an ample divisor $H_Z$
(the intersection sheaf of $(K_X)_{|Y}$ over $Z$).
In particular, the natural maps $W \to Y$ and $W \to Z$
are all $G_1$-equivariant.

Let $Z' \to Z$ be a $G_1$-equivariant resolution.
Set $Z' = Z$ when $Z$ is smooth.
Then $q(Z') = q(Y') > 0$ (cf. \cite[Lemma 5.3]{NZ}).
Since $G_1$ acts trivially on the class of the nef and big pullback on $Z'$ of $H_Z$,
we may assume that ${G_1}_{|Z'} \le \Aut_0(Z')$ after replacing $G_1$ by another finite-index subgroup of $G$,
by the result of Lieberman and Fujiki
(cf. \cite[Lemma 2.23]{JDG}, \cite[Proposition 2.2]{Li}, \cite[Theorem 4.8]{Fu}).

Since $Y'$ is not rationally connected as mentioned above,
we have $1 \le \dim Z \le \dim Y - 1 = k - 1 \le n-2$.
By \cite[Corollary 1.4]{GHS}, $Z$ is not uniruled.
Thus $\kappa(Z') \ge 0$, by the condition (4) in Theorem \ref{ThC} and since $q(Z') > 0$.

As in Lemma \ref{ample3}, we may assume that $\kappa(Z') = 0$,
by the minimality of $\dim Y = k > 0$.
By the condition (5) of Theorem \ref{ThC}, $Z'$ has a good minimal model $Z_m$
i.e., $Z_m$ has only canonical singularities and $K_{Z_m} \sim_{\Q} 0$.
Take a partial resolution $\tau: {Z_m}' \to Z_m$ such that
$Z_m'$ has only terminal singularities and $K_{{Z_m}'} = \tau^*K_{Z_m}$ ($\sim_{\Q} 0$)
(cf. \cite[Corollary 1.4.3]{BCHM}). Replacing $Z_m$ by ${Z_m}'$ we may assume that
$Z_m$ has only terminal singularities.

We may assume that
${G_1}_{| Z'}$ and hence
$$M := \Aut_0(Z')$$ are non-trivial.
The (non-trivial)
birational action of the connected algebraic group $M$ on the terminal minimal variety $Z_m$ is biregular
(cf. \cite[Corollary 3.8]{Ha}).

By the argument for $Y^n$ in Lemma \ref{ample3}, either
$M$ stabilizes every member of a covering
torus-quotient family on $Z_m$ so that the inverse on $Y$ of a general member
is $G$-periodic which contradicts the minimality of $\dim Y = k > 0$, or
$Z_m$ is a $Q$-torus with $q(Z_m) = q(Z') > 0$ and an \'etale covering
$B \to Z_m$ from an abelian variety $B$ such that the action of $M$ on $Z_m$
lifts to an action of $\widetilde{M}$ on $B$ with $\widetilde{M}/\Gal(B/Z_m) \cong M$.

We only need to consider this latter case.

\begin{claim}\label{c2}
We have $Z = Z' = Z_m$.
\end{claim}

We prove this claim.
Since $Z_m$ is a $Q$-torus, it has no rational curves, and hence
the birational map $Z' \dashrightarrow Z_m$ is actually holomorphic
(and $M$-equivariant). If $Z' \to Z_m$ is not an isomorphism,
then its non-isomorphic points
on $Z_m$ and its inverse on $B$ form Zariski-closed subsets of codimension $\ge 1$
which are respectively stabilized
by the connected algebraic groups $M_{| Z_m}$
and $\widetilde{M}_{| B} \le \Aut_0(B)$ ($=$ translations).
By \cite[Lemma 2.11]{per}, there is an $\widetilde{M}$-equivariant
quotient map $B \to B/C$ (with $C$ a subtorus of dimension in $\{0, \dots, \dim B - 1\}$)
such that every element of $\widetilde{M}_{|(B/C)} \le \Aut_0(B/C)$
($=$ translations) has a periodic point. Thus $\widetilde{M}_{| (B/C)}$ is trivial.
So a general coset of $B/C$ (or a general curve on $B$, when $\dim C = 0$),
its image in $Z_m$, and the
inverse of this image on $Y$ are respectively
stabilized by $\widetilde{M}_{|B}$, $M_{|Z_m}$ and $G_1$, contradicting the minimality of $\dim Y = k > 0$.

So we may assume that $Z' = Z_m$. If $Z' \to Z$ is not an isomorphism, then
the inverse on $B$ of its non-isomorphic points on $Z_m$ form a Zariski-closed
subset of codimension $\ge 1$ which is
stabilized by $\widetilde{M}_{|B}$. This contradicts the minimality of
$\dim Y = k > 0$ by the same argument above. The claim is proved.

\par \vskip 1pc
By Claim \ref{c2}, we have $Z = Z' = Z_m$. If $W = \Gamma_{Y/Z} \to Y$
is not isomorphic then its non-isomorphic points on $W$ maps to a $G_1$-stabilized
Zariski-closed subset of $Z$ of codimension $\ge 1$, since $W \to Z$ is equi-dimensional
and $G_1$-equivariant. This leads to a contradiction as in the proof of Claim \ref{c2} above.

Hence we may assume that
$W = Y$. By the same reasoning, we can show that
$Y$ is smooth and $W = Y \to Z = Z_m$ is smooth.

Let $\overline{G}_1 \le \Aut_0(Y)$
be the identity connected component of the closure of ${G_1}_{|Y} \le \Aut_0(Y)$ in the Zariski topology,
which is abelian because so is $G_1$.
By the minimality of $\dim Y = k > 0$,
our $Y$ has a dense open orbit $\overline{G}_1 y$ intersecting every fibre of $Y \to Z$,
and its complement $\Sigma$ in $Y$ is a finite set.
In fact, $\Sigma = \emptyset$. Otherwise, the union of fibres of $Y \to Z$ passing through
the points in $\Sigma$ is $\overline{G}_1$-stabilized and hence $G$-periodic,
contradicting the minimality of $k = \dim Y$.
Thus, $Y$ is $\overline{G}_1$-homogeneous and $Y \cong \overline{G}_1 / F$ for a subgroup
$F \le \overline{G}_1$.

If $\overline{G}_1$ is a complex torus, then so is $Y$, a contradiction to the
condition (3) in Theorem \ref{ThC}.
Therefore, the linear part $L(\overline{G}_1)$ of $\overline{G}_1$ is positive-dimensional
and a rational variety by a result of Chevalley.
Then every $L(\overline{G}_1)$-orbit on $Y$ is a unirational variety
and hence contained in a fibre of $Y \to Z$ (since the $Q$-torus $Z$ contains no
rational curves). Thus the connected linear group $L(\overline{G}_1)$
acts trivially on the base $Z$ of the fibration $Y \to Z$.
Since $L(\overline{G}_1)$ is abelian (and hence solvable),
its fixed locus (point wise) on each fibre of $Y \to Z$ is non-empty and
a proper subset of the fibre. So the union of these fixed loci on fibres
is $\overline{G}_1$-stabilized (and a proper subset of $Y$) because
$L(\overline{G}_1) \lhd \overline{G}_1$.
Thus $Y$ is not $\overline{G}_1$-homogeneous, a contradiction.
Lemma \ref{ample4} is proved.
\end{proof}

\begin{lemma}\label{K0}
Under the assumption of Theorem $\ref{ThA}(2)$ or Theorem $\ref{ThC}$,
the canonical divisor satisfies: $K_X \equiv 0$ $($numerically$)$.
\end{lemma}

\begin{proof}
By Lemmas \ref{ample3} and \ref{ample4}, $H$ is ample.
Since $g^*K_X \sim K_X$ for $g \in G$, by the proof of Lemma \ref{ample},
we have $H^{n-1} \cdot K_X = 0 = H^{n-2} \cdot K_X^2$.
Then $K_X \equiv 0$ by Lemma \ref{D=0}.
\end{proof}

\begin{setup}\label{pfThA}
{\bf Proof of Theorems \ref{ThA}(2) and \ref{ThC}}
\end{setup}

By Theorem \ref{DSThm}, $r = n-1$.
By Lemma \ref{K0}, we have $K_X \equiv 0$.
Let $m > 0$ be minimal such that $mK_X \sim 0$ (a result of Kawamata).
Replacing $X$ by its global index-one cover
$$\Spec \oplus_{i=0}^{m-1} \OO_X(-iK_X)$$
which has a natural $G$-action and whose canonical divisor is linearly equivalent to zero,
we may assume that $K_X$ ($\sim 0$) is Cartier
and hence $X$ has at worst Gorenstein canonical singularities.
Let $\sigma: X' \to X$ be a $G$-equivariant resolution crepant in codimension two.
Denote by $c_i(X')$ the $i$-th Chern class of $X'$.
As in \cite[p.~265]{SW} (cf.~also \cite[Definition 2.4]{CY3}), define the second Chern class
of $X$ as
$c_2(X) := \sigma_*c_2(X')$ and regard it as a multi-linear form on
$\NS_{\C}(X) \times \cdots \times \NS_{\C}(X)$:
$$c_2(X) \cdot D_1 \cdots D_{n-2}
:= \sigma_*c_2(X') \cdot D_1 \cdots D_{n-2} = c_2(X') \cdot \sigma^*D_1 \cdots \sigma^* D_{n-2} .$$
Since $g^*c_2(X) = c_2(X)$ for all $g \in G$,
we have the vanishing $c_2(X) \cdot H^{n-2} = 0$
({\it using $n \ge 3$ here}) as in the proof of Lemma \ref{ample}.
Since $H$ is ample by Lemmas \ref{ample3} and \ref{ample4}, this vanishing and Miyaoka's pseudo-effectivity of $c_2$
for minimal variety (cf.~\cite[Theorem 4.1, Proposition 1.1]{SW})
imply that $c_2(X) = 0$ as a multi-linear form, as remarked in \cite[Definition 2.4]{CY3}.
Now the vanishing of $c_i(X)$ ($i = 1, 2$) imply that
there is a finite surjective morphism $A' \to X$ from an abelian variety $A'$
(cf.~\cite[Corollary, p.~266]{SW}
and \cite[Theorem 7.6]{CZ}). Indeed, when $X$ has only quotient singularities, the vanishing of
$c_2(X)$ implies the vanishing of the orbifold second Chern class of $X$
(cf.~\cite[Proposition 1.1]{SW}). Since $K_{A'} \sim 0 \equiv K_X$,
the map $A' \to X$ is \'etale in codimenion one.
Let $A \to X$ be the Galois cover corresponding to the unique maximal lattice $L$
in $\pi_1(X \setminus \Sing X)$ so that $A$ is an abelian variety.
Then $X = A/F$ with $F = \pi_1(X \setminus \Sing X)/L = \Gal(A/X)$, and there is an
exact sequence
$$1 \to F \to \widetilde{G} \to G \to 1$$
where $\widetilde{G}$ (the lifting of $G$) is a group acting faithfully on $A$
(cf.~the proof in \cite[\S 3, especially Proposition 3]{Be} applied to
\'etale-in-codimension-one covers, also \cite[Proposition 3.5]{nz2}).
This proves Theorems \ref{ThA}(2) and \ref{ThC}; indeed, the last part
is the application of Lemma \ref{G/H} to the groups $F \lhd \widetilde{G}$;
see the proof of Theorem \ref{ThA}(1) below
for the freeness of the action $F$
outside a finite set.

\begin{setup}
{\bf Proof of Theorem \ref{ThA}(1)}
\end{setup}

Suppose that a positive-dimensional proper subvariety $Y' \subset X$ is $G$-periodic.
Then a proper subvariety $Y \subset A$ (dominating $Y'$) is $G$-periodic
and hence stabilized by a finite-index subgroup $G_1 \le G$.
By the proof of \cite[Lemma 2.11]{per}, there is a $G_1$-equivariant
homomorphism $A \to A/B$ with $\dim (A/B) \in \{1, \dots, n-1\}$.
This contradicts \cite[Lemma 2.10]{Tits}, because $r(G_{1|A}) = r(G_{1|X}) = \dim A - 1$ now
(cf. \cite[Appendix, Lemma A.8]{NZ}). Hence the last assertion is true.

Let $Y \subset A$ be the subset where $\tau : A \to X$ is not \'etale.
Then $Y$ is $G$-stabilized. Hence
$\dim Y = 0$ by the argument above. Thus
$K_A = \tau^*K_X$ by the ramification divisor formula.
Hence $0 \sim \tau_*K_A = \tau_* \tau^* K_X = d K_X$ with
$d = \deg(\tau)$.
Replacing $A$ by the Galois closure of $\tau$, we may assume
that $\tau : A \to X$ is Galois and hence $X$ has only quotient singularities.
This proves Theorem \ref{ThA}(1).

\begin{setup}
{\bf Proof of Corollary \ref{CorB}}
\end{setup}

By Theorem \ref{DSThm}, $r = n-1$. The `if' part follows from Theorem \ref{ThA}(2).
For the `only if' part, the proof of Theorem \ref{ThC} in \ref{pfThA} (using Lemma \ref{G/H}) implies the
lifting of a finite-index subgroup $G_1 \le G$ to some complex torus cover of $X$.
Hence $X$ has no $G_1$-periodic (i.e., $G$-periodic) subvariety $Y \subset X$ of
positive-dimension by Theorem \ref{ThA}(1). This proves the corollary.

\begin{setup}
{\bf Proof of Corollary \ref{CorD}}
\end{setup}

By Theorem \ref{DSThm}, $r = n-1$.
We have $H^{n-1} . K_X = 0$,
by the proof of Lemma \ref{ample} and since $g^*K_X \sim K_X$ for
$g \in G$. Thus $K_X \equiv 0$, by \cite[Lemma 2.2]{nz2} and since $K_X$ is nef
by the minimality of $X$.
Hence $H$ is ample by the proof of Lemma \ref{ample2}.
The rest follows from the proof of Theorem \ref{ThC} in \ref{pfThA}.

\begin{setup}
{\bf Addendum to \cite{CY3}}
\end{setup}

\cite{CY3} itself is not used in the present paper.
Thanks to the careful reading of the referee of the present paper,
in the statements of \cite[Theorem 1.1(2)]{CY3} and \cite[Corollary 1.2]{CY3} (resp. \cite[Theorem 1.5]{CY3}),
the phrase

\par \vskip 1pc \noindent
``$G$-equivariant finite (resp. finite \'etale) Galois cover $\tau: A \to X'$ (resp. $\tau: A \to X$)"
\par \vskip 1pc \noindent
should be read as:
\par \vskip 1pc \noindent
``equivariant finite (resp. finite \'etale) Galois cover $\tau: (A, \widetilde{G}) \to (X', G)$
(resp. $\tau: (A, \widetilde{G}) \to (X, G)$)
with $G \cong \widetilde{G}/\Gal(A/X')$ (resp. $G \cong \widetilde{G}/\Gal(A/X)$)" .
\par \vskip 1pc \noindent
Such lifting from $G$ on $X'$ or $X$ to $\widetilde{G}$ on $A$
is due to \cite[\S 3]{Be} applied to
\'etale-in-codimension-one covers (cf. also \cite[Proposition 3.5]{nz2}).
The proofs of \cite[Theorem 1.5 and Corollaries 1.2 and 1.6]{CY3} are not affected, while in the proof of
\cite[Theorem 1.1]{CY3}, one just replaces all $G_{|A}$ by $\widetilde{G}$, then the arguments go through.
One also notes that $N(\widetilde{G})$ is the pullback of $N(G)$ via the quotient map
$\gamma : \widetilde{G} \to G$ (cf. \cite[Remark 2.1(11)]{JDG} or \cite[Appendix A, Lemma A.8]{NZ})
and hence $\gamma$ induces an isomorphism $\widetilde{G}/N(\widetilde{G}) \cong G/N(G)$.


\begin{thebibliography}{99}

\bibitem{Be}
A.~Beauville,
Some remarks on K\"ahler manifolds with $c\sb{1}=0$,
\emph{Classification of Algebraic and Analytic Manifolds}
(Katata, 1982, ed. K.~Ueno),
Progr.\ Math., \textbf{39} Birkh\"auser 1983, pp.~1--26.

\bibitem{BCHM}
C.~Birkar, P.~Cascini, C.~D.~Hacon and J.~McKernan,
Existence of minimal models for varieties of log general type,
J. \ Amer. \ Math. \ Soc. \ \textbf{23} (2010) 405 -- 468.

\bibitem{Ca}
S.~Cantat,
Dynamique des automorphismes des surfaces projectives complexes.
C. \ R. \ Acad. \ Sci. \ Paris Sr. \ I Math. \ \textbf{328} (1999), no. 10, 901--906.

\bibitem{Ca06}
S.~Cantat,
Quelques aspects des syst\`emes dynamiques polynomiaux: 
existence, exemples et rigidit\'e (version pr\'eliminaire); Panorama et Synth\`ese, vol. 30, \`a paraitre.

\bibitem{CZ}
S.~Cantat and A.~Zeghib,
Holomorphic actions of higher rank lattices in dimension three, preprint 2009.

\bibitem{DS}
T.-C.~Dinh and N.~Sibony,
Groupes commutatifs d'automorphismes d'une vari\'et\'e k\"ahlerienne compacte,
Duke Math.\ J. \textbf{123} (2004) 311--328.

\bibitem{Fu}
A.~Fujiki,
On automorphism groups of compact K\"ahler manifolds,
Invent. Math. \textbf{44} (1978) 225--258.

\bibitem{Fujino}
O.~Fujino,
Fundamental theorems for the log minimal model program, arXiv:\textbf{0909.4445}.

\bibitem{GHS}
T.~Graber, J.~Harris and J.~Starr,
Families of rationally connected varieties,
J.\ Amer.\ Math.\ Soc.\ \textbf{16} (2003), 57--67.

\bibitem{Gr}
M.~Gromov, On the entropy of holomorphic maps,
Enseign. \ Math. \ (2) \textbf{49} (2003), no. 3-4, 217�235.

\bibitem{HM}
C.~D.~Hacon and J.~ McKernan,
On Shokurov's rational connectedness conjecture, Duke Math. \ J. \textbf{138} (2007), no. 1, 119--136.

\bibitem{Ha} M.~Hanamura,
On the birational automorphism groups of algebraic varieties,
Compos.\ Math.\ \textbf{63} (1987), 123--142.

\bibitem{Hart}
R.~Hartshorne, Algebraic Geometry, Graduate Texts in Mathematics \textbf{52}, Springer-Verlag, 1977.

\bibitem{HO}
A.~T.~Huckleberry and E.~Oeljeklaus,
A characterization of complex homogeneous cones,
Math.\ Z. \textbf{170} (1980) 181--194.

\bibitem{KMM}
Y.~Kawamata, K.~Matsuda and K.~Matsuki,
Introduction to the minimal model problem,
\emph{Algebraic geometry, Sendai, 1985} (T.~Oda ed.),
Adv.\ Stud.\ Pure Math., \textbf{10}, Kinokuniya and North-Holland,
1987, pp.~283--360.

\bibitem{KM} J.~Koll\'ar and S.~Mori,
Birational geometry of algebraic varieties,
Cambridge Tracts in Math. \textbf{134},
Cambridge Univ.\ Press, 1998.

\bibitem{Li}
D.~I.~Lieberman, Compactness of the Chow scheme: applications to
automorphisms and deformations of K\"ahler manifolds, Lecture Notes
in Mathematics, \textbf{670}, pp. 140--186, Springer, 1978.

\bibitem{IntS}
N.~Nakayama, Intersection sheaves over normal schemes,
J. Math. Soc. Japan \textbf{62} (2010) 487 -- 595.

\bibitem{NZ}
N.~Nakayama and D. -Q.~Zhang,
Building blocks of \'etale endomorphisms of complex projective manifolds,
Proc. \ London Math. \ Soc. \ \textbf{99} (2009) 725 -- 756.

\bibitem{nz2}
N.~Nakayama and D. -Q.~Zhang,
Polarized Endomorphisms of Complex Normal Varieties,
Math. \ Ann. \ \textbf{346} (2010) 991 -- 1018.

\bibitem{SW}
N.~I.~Shepherd-Barron and P.~M.~H.~Wilson,
Singular threefolds with numerically trivial first
and second Chern classes,
J.\ Alg.\ Geom.\ \textbf{3} (1994) 265--281.

\bibitem{Si}
J.~H.~Silverman, 
Rational points on $K3$ surfaces: a new canonical height, Invent. \ Math. \ \textbf{105} (1991), no. 2, 347 -- 373.

\bibitem{Ue} K.~Ueno,
\emph{Classification theory of algebraic varieties and compact complex spaces},
Lecture Notes in Mathematics, \textbf{439},
Springer, 1975.

\bibitem{We}
J.~Wehler, $K3$-surfaces with Picard number $2$, Arch. \ Math. (Basel) \textbf{50} (1988), no. 1, 73 -- 82. 

\bibitem{Tits}
D. -Q.~Zhang,
A theorem of Tits type for compact K\"ahler manifolds,
Invent. \ Math. \ \textbf{176} (2009) 449 -- 459.

\bibitem{JDG}
D. -Q.~Zhang, Dynamics of automorphisms on projective complex manifolds,
J. \ Diff. \ Geom. \ \textbf{82} (2009) 691 -- 722.

\bibitem{per}
D. -Q.~Zhang,
The $g$-periodic subvarieties for an automorphism $g$ of positive entropy on a projective variety,
Adv. \ Math. \ \textbf{223} (2010) 405 -- 415.

\bibitem{CY3}
D. -Q.~Zhang,
Automorphism groups of positive entropy on minimal projective varieties,
Adv. \ Math. \ \textbf{225} (2010) 2332 -- 2340.

\end{thebibliography}
\end{document}